\documentclass[a4paper,10pt]
{amsart}

\usepackage{amsmath,amsthm,xypic,amssymb,pdfsync,mathtools}
\usepackage[utf8]{inputenc}
\usepackage[french, german, english]{babel}

\newtheorem{thm}{Theorem}[section]

\newtheorem{prop}[thm]{Proposition}
\newtheorem{cor}[thm]{Corollary}

\theoremstyle{definition}

\theoremstyle{plain}

\author{M. G. Mahmoudi}
\date{}

\title
[Cayley parametrization and the rotation group]
{Cayley parametrization and the rotation group over a non-archimedean pythagorean field}

\begin{document}
\maketitle

\begin{abstract}
  Using Cayley transform, we show how to construct rotation matrices \emph{infinitely
  near} the identity matrix over a non-archimedean pythagorean field.
As an application, an alternative way to construct non-central proper
normal subgroups of the rotation group over such fields is provided.\\

\end{abstract}

\section{Introduction}

The Cayley transform provides an important parametrization of the rotation group of the euclidean space.
This map which was introduced by Arthur Cayley in 1846, associates a
rotation matrix to every skew symmetric matrix of eigenvalues
different from $-1$.
The Cayley parameterization was used by J. Dieudonn\'e in \cite[\S15, \S16]{dieudonne} to construct normal subgroups of the rotation group of a vector space $V$ of dimension
$n=3$ equipped with an anisotropic quadratic form over the field of $p$-adic numbers or the field of formal Laurent series in one variable over the field of real numbers for arbitrary $n\geqslant3$.
Later, using the properties of \emph{Elliptic spaces}, E. Artin in
\cite[Ch.\,V, \S3]{artin} constructed non-central subgroups of the
rotation group $\mathrm{SO}(V)$ of a vector space $V$ of dimension
$n\geqslant3$ equipped with an anisotropic quadratic $q$ form over a field
$F$ with a non-archimedean ordering.

We simplify the Artin construction using the Cayley parametrization.
The main new observations that make this simplification possible are the
following:

(1) Assume that $n$ is odd or $n$ is even and $\det q$ is trivial.
If $F$ is not formally real pythagorean then the rotation group $\mathrm{SO}(V)$ is never projectively simple (see
(\ref{trivial-spinor-norm}) and (\ref{projectively-simple})).
This enables us to concentrate on the case where the base field $F$ is pythagorean, hence the Frobenius norm is at disposal and one may consider it instead of the maximum
norm as used in \cite[Ch.\,V, \S3]{artin}.

(2)  If $F$ is a formally real pythagorean
field with a non-archimedean ordering then one can find a rotation matrix $A\in M_n(F)$ such
that the Frobenius norm of $A-I\neq0$ is infinitely small (see
(\ref{rotation-near-identity})). 

Finally, as another application, we complement\footnote{There is a gap in our proof of Corollary \ref{broecker} in the case where $\dim V$ is even. 
More precisely, if $\dim V$ is even, the truth of the statement ``if $\mathrm{SO}(V)$ is projectively simple then $F$ is pythagorean formally real field and $q$ is similar to the quadratic form $x_1^2+x_2^2+\cdots+x_n^2$'' remains unestablished.
There is a follow-up of this preprint in the Bulletin of the Iranian Mathematical Society,
February 2020, Volume 46, Issue 1, pp 253–262
} a result due to L. Br\"ocker
\cite{broecker}, which implies that the orthogonal groups $\mathrm{SO}(V)$ are all projectively simple for the case where the Witt index $\nu$ of $q$ is zero, $n\geqslant3$ and $n\neq4$ 
precisely in the case where $F$ is pythagorean formally real field, admits only of archimedean ordering and
$q$ is similar to $x_1^2+x_2^2+\cdots+x_n^2$.
This addresses the following passage by J. Dieudonn\'e in \cite[p. 39]{dieudonne}:
{\it
``
... la simplicité du groupe des rotations (pour $\nu=0$, $n>2$ et $n\neq4$)
 sur le
corps des nombres réels $\mathbb{R}$, apparaît comme un phénomène très
particulier, dû à la structure très spéciale du corps $\mathbb{R}$ parmi les corps
commutatifs. Il y aurait lieu de rechercher s'il existe d'autres corps
qui partagent avec lui cette propriété.''}.
\section{Notation and Terminology}

Throughout this paper, $F$ denotes a field of characteristic different from two.
Let $(V,b)$ be bilinear space of dimension $n$ over a field $F$ and let $q:V\rightarrow F$ be its associated nondegenerate quadratic form given by $q(x)=b(x,x)$.
The determinant of $b$ is the determinant of a Gram matrix of
$b$ modulo $F^{{\times}2}$ and is denoted by $\det b$ or $\det q$.
One says that $b$ or $q$ represents a scalar $\lambda\in F$ if there
exists a nonzero vector $v\in V$ such that $b(v,v)=\lambda$.
A vector $u\in V$ is called {\it anisotropic} (resp. {\it isotropic}) if $q(u)\neq0$ (resp. $q(u)=0$).
A quadratic form $q$ is said to be \emph{anisotropic} if all nonzero
vectors in $V$ are anisotropic, otherwise $q$ is called \emph{isotropic}.
The {\it Witt index} $\nu$ of $(V,b)$ is defined as the maximal dimension of a totally isotropic subspace of $V$.
As the characteristic of $F$ is different from $2$, there exists an
orthogonal basis $\{e_1,\cdots,e_n\}$ of $(V,b)$.
It is known that for $n\geqslant3$ the group $\mathrm{SO}(V)$ is a
non-abelian group and its center is either $\{\pm\mathrm{id}\}$ or
$\{\mathrm{id}\}$ depending on the parity of $n$ (see \cite[p. 51]{dieudonne71}).

For every $\sigma\in \mathrm{O}(V)$ by the Cartan-Dieudonn\'e theorem, there
exists a decomposition $\sigma=\prod_{i=1}^m\tau_{u_i}$ where
$m\leqslant n$ and $\tau_{u_i}$ is a reflection along
the anisotropic vectors $u_i\in V$ given by
$$\tau_{u_i}(x)=x-2\frac{b(x,u_i)}{q(u_i)}u_i.$$
It is known that the class of $\theta(\sigma)=\prod_{i=1}^nq(u_i)$ in
the quotient group $F^{\times}/F^{{\times}2}$, which is called the {\it spinor norm} of
$\sigma$, is independent of the choices and the number of $u_i$'s.

We recall that a field $F$ is said to be {\it formally real} (or {\it ordered field}) if $-1$ is not
a sum of squares in $F$.
A field $F$ is called \emph{pythagorean} if every sums of squares is
again a square.
An element $c$ of a formally real field $F$ is called \emph{totally positive} if $c$ is positive with respect to all orderings of $F$.

If $F$ is a pythagorean field then for every matrix $A$ with entries
in $F$, the \emph{Frobenius norm} of $A$, denoted by $\|A\|$, is defined as the absolute
value of the square root of the sum of squares of entries of $A$.
The Frobenius norm satisfies the inequalities
$\|AB\|\leqslant\|A\|\|B\|$ and $\|A+B\|\leqslant\|A\|+\|B\|$ when $A$ and $B$ are of appropriate sizes.

For any matrix $A\in M_n(F)$ such that $I+A$ is invertible, the map
$A\mapsto(I-A)(I+A)^{-1}$ is called the \emph{Cayley map}.
For any skew-symmetric matrix $A\in M_n(F)$ such that the
matrix $I+A$ is invertible, the matrix $Q=C(A)$ satisfies $Q^TQ=I$
and $\det Q=1$. 
Conversely for every orthogonal matrix $Q$ such that $I+Q$ is
invertible $C(Q)$ is a skew-symmetric matrix, see \cite[p. 56]{weyl}.

\section
{Rotation group over non-archimedean fields}

As a first observation we have the following result:

\begin{prop}
\label{trivial-spinor-norm}
Let $(V,b)$ be anisotropic bilinear space of dimension $n\geqslant2$
over a field $F$.
Then the image of the spinor norm $\theta:\mathrm{SO}(V)\rightarrow F^\times/F^{\times 2}$ is trivial if and only if $F$ is a formally real pythagorean field and $q$ represents only one square class in $F^\times$, in particular $q$ is similar to the quadratic form $x_1^2+\cdots+x_n^2$.
\end{prop}

\begin{proof}
 First assume that the spinor norm map is the trivial map.
 Let $u\in V$ be an anisotropic vector with $q(u)=d\neq0$.
Then for any anisotropic vector $v\in V$ as
$\theta(\tau_u\tau_v)$ is trivial, the scalars $q(u)$ and
$q(v)$ are in the same square class in $F^{{\times}}$.
It follows that $q$ is isomorphic to 
$dx_1^2+\cdots+dx_n^2$.
Let $\alpha, \beta\in F$, we should prove that there exists $\gamma\in
F$ such that $\gamma^2=\alpha^2+\beta^2$.
As $n\geqslant2$, we may consider two orthogonal vectors $u, v\in V$ with
$q(u)=q(v)=d\neq0$.
Now consider the vector $w=\alpha u+\beta v\in V$.
We have $q(w)=(\alpha^2+\beta^2)d$.
As $q(w)$ and $q(u)$ are in the same square classes, the quantity
$\alpha^2+\beta^2$ is a square in $F$.
If $F$ is not formally real, then $-1$ is a sum of squares in $F$,
hence is a square since $F$ is pythagorean. 
As $q\simeq dx_1^2+\cdots+dx_n^2$ and $n\geqslant2$, the form $q$
would be isotropic, contradiction. 
The converse follows from the Cartan-Dieudonn\'e theorem and the fact that the current hypotheses imply that the spinor norm of the product of two arbitrary reflections is trivial.
\end{proof}

\begin{prop}
\label{projectively-simple}
Let $(V,b)$ be an anisotropic bilinear space of dimension $n\geqslant3$ over a field $F$ such that $\mathrm{SO}(V)$ is projectively simple.
Then the spinor norm map $\theta:\mathrm{SO}(V)\rightarrow F^\times/F^{\times2}$ is the trivial map precisely in the following cases $(i)$ $n$ is odd,
$(ii)$ $n$ is even and the determinant of $b$ is trivial.
\end{prop}

\begin{proof}
First suppose that $\theta$ is the trivial map.
Let $\{e_1,\cdots,e_n\}$ be an orthogonal basis of $(V,b)$. 
We have $-\mathrm{id}=\tau_{e_1}\tau_{e_2}\cdots\tau_{e_n}$.
It follows that in the case where $n$ is even, $-\mathrm{id}\in\mathrm{SO}(V)$
 and
$\theta(-\mathrm{id})$ which coincides with the determinant of $b$, is trivial.

Conversely suppose that $(i)$ or $(ii)$ hold.
In the case $(i)$ the group $\mathrm{SO}(V)$ has trivial center.
Thus the hypotheses actually say that $\mathrm{SO}(V)$ is itself a simple group.
It follows that $\mathrm{ker}(\theta)$ is either the trivial group or is
the whole $\mathrm{SO}(V)$.
The first case is impossible as $\mathrm{SO}(V)$ is non-abelian.
The second case implies that $\mathrm{im}(\theta)$ is trivial.
In the case $(ii)$ the center of $\mathrm{SO}(V)$ is the subgroup $\{\pm\mathrm{id}\}$.
As $-\mathrm{id}=\tau_{e_1}\tau_{e_2}\cdots\tau_{e_n}$ and the
determinant of $b$ is assumed to be trivial, $\theta(-\mathrm{id})$ is
trivial as well.
Therefore, the kernel of the spinor norm map $\theta$ is a subgroup of $\mathrm{SO}(V)$
containing $\{\pm\mathrm{id}\}$.
As $\mathrm{SO}(V)$ is assumed to be projectively simple,
$\mathrm{ker}(\theta)$ is either $\{\pm\mathrm{id}\}$, or
$\mathrm{SO}(V)$.
If $\mathrm{ker}(\theta)=\{\pm\mathrm{id}\}$, we have
$\mathrm{SO}(V)/\{\pm\mathrm{id}\}\simeq\mathrm{im}(\theta)$.
This shows that $\mathrm{SO}(V)/\{\pm\mathrm{id}\}$ is an abelian simple group.
  It follows that this quotient group is either the trivial group
(thus $\mathrm{im}(\theta)$ is the trivial group) or a group of prime
order.
The later case is ruled out as a non-abelian group is never projectively cyclic.
The former case also implies the
triviality of $\mathrm{im}(\theta)$.
\end{proof}

\begin{prop}
\label{rotation-near-identity}
  Let $F$ be a pythagorean field with a non-archimedean ordering $\leqslant$.
Then there exists a matrix $A\in M_n(F)$ with $A\neq\pm I$, $A^TA=I$ and $\det A=1$
such that $\|I-A\|$ is infinitely small.
\end{prop}

\begin{proof}
  Let $\epsilon$ be an infinitely small positive element of $F$ and
  let $B\in M_n(F)$ be a nonzero skew-symmetric matrix whose entries
  are rational.
Consider the matrix $A=C(\epsilon B)$, where $C$ is the Cayley map.
We claim that $\|I-A\|$ is infinitely small.
First note that $I+\epsilon B$ is invertible, hence $C(\epsilon B)$ is
meaningful and is different from $\pm I$.
Let $m$ be an odd positive integer.
The relation $I+\epsilon^mB^m=(I+\epsilon B)D$ where $D=(I-\epsilon
B+\epsilon^2B^2+\cdots+\epsilon^{m-1}B^{m-1})$ implies that $\|(I+\epsilon
B)^{-1}-D\|=\epsilon^m\|B^m(I+\epsilon B)^{-1}\|$.
By the Cramer rule every entry of $B^m(I+\epsilon B)^{-1}$ is of the
form 
${p(\epsilon)}/{q(\epsilon)}$ where $p(X),
q(X)\in\mathbb{Q}[X]$ are of degree at most $n$.
It follows that for sufficiently large $m$, the quantity
$\|(I+\epsilon B)^{-1}-D\|=\epsilon^m p(\epsilon)/q(\epsilon)$ is infinitely small.
Hence
\begin{align*}
  \|I-A\|=&\|I-(I-\epsilon B)(I+\epsilon B)^{-1}\|\\
\leqslant&\|I-(I-\epsilon B)D\|+\|(I-\epsilon B)((I+\epsilon
           B)^{-1}-D)\|\\
\leqslant& \|I-(I-\epsilon B)D\|+\|I-\epsilon B\|\|(I+\epsilon
           B)^{-1}-D\|
\end{align*}
Both quantities $\|(I-(I-\epsilon B))D\|$ and $\|(I+\epsilon
B)^{-1}-D\|$ are infinitely small and the proof is complete.
\end{proof}

\begin{prop}
\label{archimedean-is-necessary}
 Let $F$ be a pythagorean ordered field and let $(V,b)$ be a bilinear
 space of dimension $n\geqslant3$, isometric to $x_1^2+\cdots+x_n^2$.
If $F$ carries a non-archimedean ordering then $\mathrm{SO}(V)$ contains a
proper non-central normal subgroup.
\end{prop}

\begin{proof}
We claim that the group 
$$N=\{\sigma\in\mathrm{SO}(V): \ \|x-\sigma(x)\| \textrm{ is
 infinitely small if } q(x)=1\}$$ is a non-trivial
normal subgroup of $\mathrm{SO}(V)$ (the idea of considering this subgroup was borrowed from \cite[p. 150]{perrin}).
To prove this claim we first show that $N$ is a subgroup of $\mathrm{SO}(V)$.
Consider two elements $\sigma, \tau\in N$.
We should prove that $\sigma\tau$ also belongs to $N$.
We have
$\|x-\sigma\tau(x)\|=\|(x-\tau(x))+(\tau(x)-\sigma\tau(x))\|\leqslant\|x-\tau(x)\|+\|\tau(x)-\sigma\tau(x)\|=\|x-\tau(x)\|+\|x-\sigma(x)\|$,
which is infinitely small.
The fact that $\sigma^{-1}\in N$ when $\sigma\in N$ and the normality
of $N$ is straightforward.
It remains to show that $N$ is a nontrivial subgroup of
$\mathrm{SO}(V)$.
Consider two orthogonal vectors $u, v\in V$ with $q(u)=q(v)=1$.
Let $\sigma=-\tau_u\in\mathrm{SO}(V)$.
We have $\|v-\sigma(v)\|=\|v+\tau_u(v)\|=\|2v\|=2$.
Hence $\sigma\not\in N$, thus $N\neq\mathrm{SO}(V)$.
By (\ref{rotation-near-identity}), there exists a rotation matrix $A$
such that $\|I-A\|$ is infinitely small. 
It follows that for every $x$ with $\|x\|=1$ we have
$\|x-Ax\|\leqslant\|I-A\|$ is infinitely small.
Hence $A\neq\pm I$ is an element of $N$ and $N\neq\{\pm I\}$
\end{proof}

\begin{cor}
  \label{broecker}
Let $(V,b)$ be a bilinear space of dimension $n\geqslant3$  over a
field $F$ whose Witt index is zero.
Then the rotation group $\mathrm{SO}(V)$ is a projectively simple if and only if $F$ is a pythagorean formally real field and $q$ is similar to the quadratic form $x_1^2+x_2^2+\cdots+x_n^2$.
\end{cor}

\begin{proof}
In \cite{broecker}, Br\"ocker proves that
if $F$ is a pythagorean formally real field which admits only of
archimedean orderings, $n\geqslant3$ and $n\neq4$ and $q$ is anisotropic then the kernel of the spinor norm
$\theta:\mathrm{SO}(V)\rightarrow F^\times/F^{\times2}$ coincides with
the commutator subgroup $\Omega(V)$ of the orthogonal group
$\mathrm{O}(V)$ (see \cite[(1.7)]{broecker}) and the group $\Omega(V)$ is projectively simple (see \cite[Satz (1.10)]{broecker}).
Of course when $q$ represents only one square class, the whole rotation group $\mathrm{SO}(V)$ coincides with the kernel of the
spinor norm.
Hence the sufficiency of the conditions can be obtained from
Br\"ocker's theorem.
The sufficiency follows from (\ref{trivial-spinor-norm}), (\ref{projectively-simple}) and
(\ref{archimedean-is-necessary}).
\end{proof}

\noindent
 \textbf{Acknowledgments.}
The support from Sharif University of Technology is gratefully acknowledged.

\footnotesize
\noindent{\sc M. G. Mahmoudi, {\tt
    mmahmoudi@sharif.ir}, \\  Department of Mathematical Sciences, Sharif University of Technology, P. O. Box 11155-9415, Tehran, Iran.}


\begin{thebibliography}{99}

\bibitem{artin}
Artin, E.
\emph{Geometric algebra.}
Interscience tracts in pure and applied mathematics. No. 3. New York: Interscience Publishers, Inc.; London: Interscience Publishers Ltd.  (1957).



\bibitem{broecker}
Bröcker, L.
Zur orthogonalen Geometrie über pythagoreischen Körpern.
J. Reine Angew. Math. {\bf 268/269}, 68-77 (1974).


\bibitem{dieudonne}
 Dieudonn\'e, J.
\emph{Sur les groupes classiques. }
Publications de l'Institut de Math\'ematique de l'Universit\'e de Strasbourg, VI. Actualit\'es Scientifiques et Industrielles, No. 1040. Hermann, Paris, 1973.

\bibitem{dieudonne71}
Dieudonné, J. 
\emph{La géométrie des groupes classiques. }
Ergebnisse der Mathematik und ihrer Grenzgebiete. Band
5. Berlin-Heidelberg-New York: Springer-Verlag (1971).









\bibitem{perrin}
Perrin, D.
{\it Cours d'alg\`ebre}, Paris: Ellipes (1996).



\bibitem{weyl}
Weyl, H.
\emph{The classical groups, their invariants and representations}. Reprint of the second edition (1946) of the 1939 original.
Princeton University Press, Princeton, N.J. (1939).
\end{thebibliography}
\end{document}